\documentclass{amsart}

\usepackage[inactive]{srcltx} 
\usepackage{amsmath, amsthm, amscd, amsfonts, amssymb, graphicx, color, extarrows}
 \newtheorem{thm}{Theorem}[section]
 
 \newtheorem{lem}[thm]{Lemma}
 
 \theoremstyle{definition}
 
 \theoremstyle{remark}
 \newtheorem{rem}[thm]{Remark}
 \numberwithin{equation}{section}
\newcommand{\g}{L^{1}(G)}
\newcommand{\m}{M(G)}

\begin{document}

\title[characterizing derivations and anti-derivations on group algebras...]
 {characterizing derivations and anti-derivations on group algebras through orthogonality}

\author{Hoger Ghahramani}
\thanks{{\scriptsize
\hskip -0.4 true cm \emph{MSC(2010)}: 22D15; 47B47; 15A86. 
}}

\address{Department of
Mathematics, University of Kurdistan, P. O. Box 416, Sanandaj,
Iran.}

\email{h.ghahramani@uok.ac.ir; hoger.ghahramani@yahoo.com}


\address{}

\email{}

\thanks{}

\thanks{}



\date{}

\dedicatory{}

\commby{}

\begin{abstract}
Let $\g$ and $\m$ be group algebra and measure algebra, respectively of a locally compact group $G$ and $\Delta:\g \rightarrow \m$ be a continuous linear map. We consider $\Delta$ behaving like derivation or anti-derivation at orthogonal elements for several types of orthogonality conditions and we characterize such maps. Indeed we consider that $\Delta$ is a derivation or anti-derivation through orthogonality conditions on $\g$ such as $f*g=0$, $f*g^{\star}=0$, $f^{\star}*g=0$, $f*g=g*f=0$ and $f*g^{\star}=g^{\star}*f=0$.
\end{abstract}

\maketitle
\section{ Introduction}
Throughout this paper all algebras and vector spaces will be over the complex field
$\mathbb{C}$. Let $A$ be an algebra and $M$ be an $A$-bimodule. Recall that a linear map $D:A\rightarrow M$ is said a \textit{derivation} if $D(ab) = aD(b)+D(a)b $ for all $a, b \in A$. Each map of the form
$a \mapsto am-ma$, where $m \in M$, is a derivation which will be called an \textit{inner
derivation}. Also $D$ is called an \textit{anti-derivation} if $D(ab) =bD(a)+D(b)a$ for all $a,b \in A$. There have been a number of papers concerning the study of conditions under which mappings of (Banach) algebras can be completely determined by the action on some sets of points. We refer the reader to \cite{al1, al10, bre, chb, gh1, gh2} for a full account of the topic and a list of references. In the case of derivations, the subsequent condition attracted much attention of some mathematicians:
\[a, b \in A, \,\, ab=z\Rightarrow \Delta(ab) = a\Delta(b)+\Delta(a)b \quad (\blacklozenge),\]
where $z\in A$ is a fixed point and $\Delta:A\rightarrow M$ is a linear (additive) map. Bre$\check{\textrm{s}}$ar, \cite{bre} study the derivations of rings with idempotents in this direction with $z=0$. It was shown in \cite{bre} that if $A$ is a prime ring containing a nontrivial idempotent and $\Delta:A\rightarrow A$ is an additive map satisfying $(\blacklozenge)$ with $z=0$, then $\Delta(a)=D(a)+ca$ ($a\in A$) where $D$ is an additive derivation and $c$ is a central element of $A$. Note that the nest algebras are important operator algebras that are not prime. Jing et al. in \cite{jin} showed that, for the cases of nest algebras on a Hilbert space and standard operator algebras in a Banach space, the set of linear maps satisfying $(\blacklozenge)$ with $z=0$ and $\Delta(I)=0$ coincides with the set of inner derivations. Then, many studies have been done in this case and different results were obtained, for instance, see \cite{al1, al10, bre, chb, che,  gh1, gh12} and the references therein.
\par 
The other direction is to study linear (additive) maps that behave like homomorphisms of (Banach) algebras when acting on special products. Especially, one of the interesting question is the characterizing linear maps of group algebras and other Banach algebras associated with a locally compact group behaving like homomorphisms at zero product elements or orthogonal elements. This question has been extensively studied \cite{al1, al10, al11, al111, al2, fo, fo2, la, lin, mo}. 
\par 
Motivated by these reasons, in this paper we consider the problem of characterizing continuous linear maps on group algebras behaving like derivations or anti-derivations at orthogonal elements for several types of orthogonality conditions. In particular, in this paper we consider the subsequent conditions on a continuous linear map $\Delta$ from a group algebra $\g$ into the measure convolution algebra $\m$ where $G$ is a locally compact group:
\begin{enumerate}
\item[(i)] \textit{derivations through one-sided orthogonality conditions}
\begin{enumerate}
\item[(D1)] \[ f*g=0 \Longrightarrow f*\Delta(g)+\Delta(f)*g=0\quad (f,g \in \g);\]
\item[(D2)]  \[f*g^{\star}=0 \Longrightarrow f*\Delta(g)^{\star}+\Delta(f)*g^{\star}=0;\]
\item[(D3)]  \[f^{\star}*g=0 \Longrightarrow f^{\star}*\Delta(g)+\Delta(f)^{\star}*g=0;\]
\end{enumerate}
\item[(ii)] \textit{anti-derivations through one-sided orthogonality conditions}
\begin{enumerate}
\item[(D4)] \[ f*g=0 \Longrightarrow g*\Delta(f)+\Delta(g)*f=0;\]
\item[(D5)] \[ f*g^{\star}=0 \Longrightarrow \Delta(g)^{\star}*f+g^{\star}*\Delta(f)=0;\]
\item[(D6)] \[f^{\star}*g=0 \Longrightarrow \Delta(g)*f^{\star}+g*\Delta(f)^{\star}=0;\]
\end{enumerate}
\item[(iii)] \textit{Derivations through two-sided orthogonality conditions}
\begin{enumerate}
\item[(D7)] \[ f*g=g*f=0 \Longrightarrow f*\Delta(g)+\Delta(f)*g=g*\Delta(f)+\Delta(g)*f=0;\]
\item[(D8)]
\[ f*g^{\star}=g^{\star}*f=0 \Longrightarrow f*\Delta(g)^{\star}+\Delta(f)*g^{\star}=\Delta(g)^{\star}*f+g^{\star}*\Delta(f)=0;\]
\end{enumerate}
\end{enumerate}
where $f,g \in \g$, the convolution product is denoted by $*$ and the involution is denoted by $\star$. It is worth noting that the conditions $D1$ and $D4$, $D2$ and $D3$, $D5$ and $D6$, $D7$ and $D8$ agree in the case where the group $G$ is abelian.
\par 
Our purpose is to investigate whether the above conditions characterize 
derivations ($\star$-derivations) or anti-derivations ($\star$-anti-derivations). This article is organized as follows. In section 2 some preliminaries are given. Section 3 is concerned with characterizing derivations and anti-derivations through one-sided orthogonality conditions (conditions $D1-D6$). In the last section continuous linear maps of group algebras of a $SIN$ group satisfying in conditions $D7$ and $D8$ (derivations through two-sided orthogonality conditions) are considered.
\par 
We note that the centre of an algebra $A$ are written by $\mathcal{Z}(A)$.

\section{ Preliminaries}
Let $G$ be a locally compact group. The \textit{group algebra} and the \textit{measure convolution algebra} of $G$, are denoted by $\g$ and $\m$, respectively. The convolution product is denoted by $*$ and the involution is denoted by $\star$. The element $\delta_{e}$ is the identity of $\m$, where $\delta_{e}$ is the point mass at $e\in G$ and $e$ is the identity of $G$. The measure algebra $\m$ is a unital Banach $\star$-algebra, and $\g$ is a closed ideal in $\m$, identified with the subspace of $\m$ consisting of measures which are absolutely continuous with respect to the Haar measure. If a net $(\lambda_{i})_{i\in I}$ in $\m$ converges to $\lambda \in \m$ with respect to the weak$^{*}$ topology, we write it by $\lambda_{i}\xlongrightarrow{w^{*}} \lambda$. Every group algebra $\g$ has a bounded approximate identity. The group $G$ is a $SIN$ \textit{group} if it has a base of compact neighborhoods of the the identity that is invariant under all inner automorphisms. If $G$ is a $SIN$ group, we denote it by $G\in [SIN]$. It is known that the group algebra $\g$ has a bounded approximate identity consisting of functions in $\mathcal{Z}(\g)$ if and only if $G\in [SIN]$. We refer the reader to \cite [Section 3.3]{da} for the essential information about the group algebras and measure algebras. Also see \cite [section 12.5 and 12.6]{pa} for a discussion of the class of $SIN$ groups.
\par 
In order to prove our results we need the following results.
\begin{lem}(\cite [Lemma 1.1] {al2}).\label{bi}
Let $G$ be a locally compact group, and let $\phi: \g \times \g\rightarrow X$ be a continuous bilinear map, where $X$ is a Banach space. 
\begin{enumerate}
\item[(i)] Suppose that 
\[ f,g \in \g, \, \, f*g=0\Longrightarrow \phi(f,g)=0.\]
Then 
\[ \phi(f*g,h)=\phi(f,g*h),\]
for all $f,g\in\g$.
\item[(ii)] Suppose that 
\[ f,g \in \g, \, \, f*g=g*f=0\Longrightarrow \phi(f,g)=0.\]
Then 
\[ \phi(f*g,h)=\phi(f,g*h),\]
for all $f,h\in \mathcal{Z}(\g)$ and $g\in\g$, and
\[ \phi(f*g,h*k)-\phi(f,g*h*k)+\phi(k*f,g*h)-\phi(k*f*g,h)=0,\]
for all $f,g,h,k\in\g$.
\end{enumerate}
\end{lem}
\begin{lem}(\cite [Lemma 1.3] {al2}).\label{zz}
Let $G$ be a locally compact group, and let $\mu\in \m$.
\begin{enumerate}
\item[(i)] Suppose that $\mu*\g=\lbrace0\rbrace$. Then $\mu=0$.
\item[(ii)] Suppose that $\mu*f=f*\mu$ for each $f\in \g$. Then $\mu \in \mathcal{Z}(\m)$.
\end{enumerate}
\end{lem}
Note that by \cite [Theorem 6.3]{ke} the convolution product in  $\m$ is separately continuous with respect to the weak$^{*}$ topology, i.e., $\nu\mapsto \mu*\nu$ is $w^{*}$-continuous for each $\mu\in\m$ and $\mu\mapsto \mu*\nu$ is $w^{*}$-continuous for each $\nu\in\m$.
\begin{rem}\label{we}
Let $(u_{i})_{i\in I}$ be a bounded approximate identity of $\g$. Since $(u_{i})_{i\in I}$ is bounded, we can assume that it converges to $\mu \in \m$ with respect to the weak$^{*}$ topology. So by separately $w^{*}$-continuity of convolution product in $\m$ we have $u_{i}*f\xlongrightarrow{w^{*}} \mu* f$ for all $f\in \g$. On the other hand by the fact that $(u_{i})_{i\in I}$ is an approximate identity, for each $f\in \g$ we get $u_{i}*f\xlongrightarrow{w^{*}} f$ in $\m$. So $(\mu-\delta_{e})*\g=\lbrace 0 \rbrace $ and by Lemma \ref{zz}-$(i)$, it follows that $\mu=\delta_{e}$. Therefore we can assume that the group algebra $\g$ has a bounded approximate identity such that $u_{i}\xlongrightarrow{w^{*}} \delta_{e}$ in $\m$.
\end{rem}
Let $D:\g\rightarrow\m$ be a map. We say that $D$ is a $\star$-map whenever $D(f^{\star})=D(f)^{\star}$ for all $f\in \g$.
\begin{rem}\label{der}
Let $D:\g\rightarrow \m$ be a continuous derivation. According to \cite{lo} (derivation problem), there exists $\mu\in \m$ such that $D(f)=f*\mu-\mu*f$ for any $f\in\g$. If $D$ is a continuous $\star$-derivation, then $D(f^{\star})=D(f)^{\star}$ and hence $f^{\star}*\mu-\mu*f^{\star}=\mu^{\star}*f^{\star}-f^{\star}*\mu^{\star}$ for all $f\in\g$. So by Lemma \ref{zz}-$(ii)$ we have $Re\mu=\dfrac{1}{2}(\mu+\mu^{\star})\in \mathcal{Z}(\m)$. Conversely for $\mu\in \m$ with $Re\mu\in \mathcal{Z}(\m)$, the map $D:\g\rightarrow \m$ defined by $D(f)=f*\mu-\mu*f$ is a continuous $\star$-derivation.
\end{rem}
Let $A$ be an algebra. Recall that a linear map $D:A\rightarrow A$ is said to be a \textit{Jordan derivation} if $D(a^{2})=aD(a)+D(a)a$ for all $a\in A$. Clearly, each derivation is a Jordan derivation. The converse is, in
general, not true. Sinclair \cite{sin} shows that a continuous Jordan derivation on a semisimple Banach algebra is a derivation. Since $\g$ is a semisimple Banach algebra, it follows that any continuous Jordan derivation $D:\g\rightarrow \g$ is a derivation.
\section{Derivations and anti-derivations through one-sided orthogonality conditions}
In this section we will consider a linear map $\Delta:\g \rightarrow \m$ behaving like derivation or anti-derivation at one-sided orthogonality conditions. firstly we characterize derivations through one-sided orthogonality conditions. 
\begin{thm}\label{o1}
Let $G$ be a locally compact group, and let $\Delta:\g \rightarrow \m$ be a continuous linear map.
\begin{enumerate}
\item[(i)] Assume that 
\begin{equation*}
 f*g=0 \Longrightarrow f*\Delta(g)+\Delta(f)*g=0\quad (f,g \in \g).
\end{equation*}
Then there are $\mu, \nu\in \m$ such that $\Delta(f)=f*\mu-\nu*f$ for all $f \in \g$ and $\mu-\nu\in \mathcal{Z}(\m)$.
\item[(ii)] Assume that 
\begin{equation*}
 f*g^{\star}=0 \Longrightarrow f*\Delta(g)^{\star}+\Delta(f)*g^{\star}=0\quad (f,g \in \g).
 \end{equation*}
Then there are $\mu, \nu\in \m$ such that $\Delta(f)=f*\mu-\nu*f$ for all $f \in \g$ and $Re\mu\in \mathcal{Z}(\m)$.
\item[(iii)] Assume that
\begin{equation*}
 f^{\star}*g=0 \Longrightarrow f^{\star}*\Delta(g)+\Delta(f)^{\star}*g=0\quad (f,g \in \g).
 \end{equation*}
Then there are $\mu, \nu\in \m$ such that $\Delta(f)=f*\nu-\mu*f$ for all $f \in \g$ and $Re\mu\in \mathcal{Z}(\m)$.
\end{enumerate}
\end{thm}
\begin{proof}
$(i)$ By \cite [Theorem 4.6]{al1} and Lemma \ref{zz}-$(i)$, there is a continuous derivation $D:\g \rightarrow \m$ and a measure $\xi \in \mathcal{Z}(\m)$ such that $\Delta(f)=D(f)+\xi*f$ for all $f \in \g$. From derivation problem $D(f)=f*\mu-\mu*f$ for all $f\in \g$, where $\mu\in \m$. Setting $\nu=\mu-\xi$. So $\Delta(f)=f*\mu-\nu*f$ for all $f \in \g$ and $\mu-\nu\in \mathcal{Z}(\m)$.
\par 
$(ii)$ Suppose that $(u_i)_{i\in I}$ is a bounded approximate identity of $\g$ such that $u_{i}\xlongrightarrow{w^{*}} \delta_{e}$, where $\delta_{e}$ is the identity of $\m$. Since the net $(\Delta(u_{i}))_{i\in I}$ is bounded, we can assume that it converges to $\xi \in \m$ with respect to the weak$^{*}$ topology. Define $D:\g\rightarrow \m$ by $D(f)=\Delta(f)-\xi*f$. Then $D$ is a continuous linear map which
satisfies 
\begin{equation}\label{1o4}
 f*g^{\star}=0 \Longrightarrow f*D(g)^{\star}+D(f)*g^{\star}=0\quad (f,g \in \g),
 \end{equation}
and $D(u_{i})\xlongrightarrow{w^{*}} 0$. We will show that $D$ is a  $\star$-derivation. In order to prove this we consider the continuous bilinear map $\phi:\g \times \g \rightarrow \m$ by $\phi(f,g)=f*D(g^{\star})^{\star}+D(f)*g$. If $f,g\in \g$ are such that $f*g=0$, then $f*(g^{\star})^{\star}=0$ and \eqref{1o4} gives $\phi(f,g)=0$. So by Lemma \ref{bi}-$(i)$, we get $ \phi (f*g,h)=\phi (f,g*h)$ for all $f,g,h \in \g$. Therefore
\begin{equation}\label{1o5}
f*g*D(h^{\star})^{\star}+D(f*g)*h=f*D(h^{\star}*g^{\star})^{\star}+D(f)*g*h,
 \end{equation}
 for all $f,g,h \in \g$. On account of \eqref{1o5}, for all $g,h \in \g$ we have
 \[u_{i}*g*D(h^{\star})^{\star}+D(u_{i}*g)*h=u_{i}*D(h^{\star}*g^{\star})^{\star}+D(u_{i})*g*h.\]
From continuity of $D$, we get $u_{i}*D(h^{\star}*g^{\star})^{\star}+D(u_{i})*g*h $ converges to $g*D(h^{\star})^{\star}+D(g)*h$ with respect to the norm topology. On the other hand, from separately $w^{*}$-continuity of convolution product in $\m$ and $D(u_{i})\xlongrightarrow{w^{*}} 0$ it follows that
$u_{i}*D(h^{\star}*g^{\star})^{\star}+D(u_{i})*g*h \xlongrightarrow{w^{*}}D(h^{\star}*g^{\star})^{\star}$. Hence 
\begin{equation}\label{1o6}
D(f*g^{\star})=D(f)*g^{\star}+f*D(g)^{\star},
 \end{equation}
for all $f,g \in \g$. Now letting $f=u_{i}$ in \eqref{1o6}, we obtain
\[D(u_{i}*g^{\star})=D(u_{i})*g^{\star}+u_{i}*D(g)^{\star},\]
for all $g\in \g$. By continuity of $D$, $D(u_{i})\xlongrightarrow{w^{*}} 0$ and using similar arguments as above it
follows that $D(g^{\star})=D(g)^{\star}$ for all $g\in \g$. Thus $D$ is a $\star$-derivation and by Remark \ref{der}, there is a $\mu\in \m$ with $Re\mu\in \mathcal{Z}(\m)$, such that $D(f)=f*\mu-\mu*f$ for all $f\in \g$. Setting $ \nu=\mu-\xi$. From definition of $D$, we arrow at $\Delta(f)=f*\mu-\nu*f$ for all $f \in \g$ where $Re\mu\in \mathcal{Z}(\m)$.
\par 
$(iii)$ Consider the map $D:\g\rightarrow \m$ defined by $D(f)=\Delta(f^{\star})^{\star}$. It is easily seen that the map $D$ satisfies
\[  f*g^{\star}=0 \Longrightarrow f*D(g)^{\star}+D(f)*g^{\star}=0\quad (f,g \in \g).\]
By $(ii)$, there exists $\mu_{1}, \nu_{1}\in \m$ such that $D(f)=f*\mu_{1}-\nu_{1}*f$ for all $f \in \g$ and $Re\mu_{1}\in \mathcal{Z}(\m)$. Then $\Delta(f)=f*\nu-\mu*f$ for all $f\in \g$, where $\nu=-\nu_{1}^{\star}$, $\mu=-\mu_{1}^{\star}$ and $Re\mu\in \mathcal{Z}(\m)$.
\end{proof}
In the next theorem we characterize anti-derivations through one-sided orthogonality conditions. 
\begin{thm}\label{anti}
Let $G$ be a locally compact group, and let $\Delta:\g \rightarrow \m$ be a continuous linear map.
\begin{enumerate}
\item[(i)] Assume that 
\[ f*g=0 \Longrightarrow g*\Delta(f)+\Delta(g)*f=0\quad (f,g \in \g).\]
Then there are measures $\mu,\nu \in \m$ such that $\Delta(f)=f*\mu-\nu*f$, where $\mu-\nu\in \mathcal{Z}(\m)$ and 
\[ [[f,g],\mu]+2 [f,g]*(\mu-\nu)=0, \]
for all $f,g\in \g$.
\item[(ii)] Assume that 
\begin{equation*}
 f*g^{\star}=0 \Longrightarrow \Delta(g)^{\star}*f+g^{\star}*\Delta(f)=0\quad (f,g \in \g).
 \end{equation*}
Then there are $\mu, \nu\in \m$ such that $\Delta(f)=f*\nu-\mu*f$, where $Re\mu\in \mathcal{Z}(\m)$ and 
\[ [[f,g],\mu]+(\nu-\mu)^{\star}*[f,g]+[f,g]*(\nu-\mu)=0,\]
for all $f,g \in \g$ .
\item[(iii)] Assume that 
\begin{equation*}
 f^{\star}*g=0 \Longrightarrow \Delta(g)*f^{\star}+g*\Delta(f)^{\star}=0\quad (f,g \in \g).
 \end{equation*}
Then there are $\mu, \nu\in \m$ such that $\Delta(f)=f*\mu-\nu*f$, where $Re\mu\in \mathcal{Z}(\m)$ and 
\[ [[f,g],\mu]+[f,g]*(\mu-\nu)^{\star}+(\mu-\nu)[f,g]=0,\]
for all $f,g \in \g$ .
\end{enumerate}
\end{thm}
\begin{proof}
Suppose that $(u_i)_{i\in I}$ is a bounded approximate identity of $\g$ such that $u_{i}\xlongrightarrow{w^{*}}\delta_{e}$, where $\delta_{e}$ is the identity of $\m$.
\par 
$(i)$ Define a continuous bilinear map $\phi:\g \times \g \rightarrow \m$ by $\phi(f,g)=g*\Delta(f)+\Delta(g)*f$. Then $\phi (f,g) = 0$ for all $f,g \in \g$ with $f*g=0$. By applying Lemma \ref{bi}-$(i)$, we obtain $ \phi (f*g,h)=\phi (f,g*h)$ for all $f,g,h \in \g$. So 
\begin{equation}\label{anti1}
h*\Delta(f*g)+\Delta(h)*f*g= g*h*\Delta(f)+\Delta(g*h)*f,
\end{equation}
for all $f,g,h \in \g$. Since the net $(\Delta(u_{i}))_{i\in I}$ is bounded, we can assume that it converge to $\xi \in \m$ with respect to the weak$^{*}$ topology. On account of \eqref{anti1}, for all $f,g \in \g$ we have
\[ u_{i}*\Delta(f*g)+\Delta(u_{i})*f*g= g*u_{i}*\Delta(f)+\Delta(g*u_{i})*f.\]
From continuity of $\Delta$, we get $ u_{i}*\Delta(f*g)+\Delta(u_{i})*f*g$ converges to $g*\Delta(f)+\Delta(g)*f$ with respect to the norm topology. On the other hand, by separately $w^{*}$-continuity of convolution product in $\m$, it follows that $ u_{i}*\Delta(f*g)+\Delta(u_{i})*f*g$ converges to $\Delta(f*g)+\xi*f*g$ with respect to the weak$^{*}$ topology. Hence
\begin{equation}\label{anti2}
\Delta(f*g)=g*\Delta(f)+\Delta(g)*f-\xi*f*g
\end{equation}
for all $f,g,h \in \g$. Now letting $f=u_{i}$ in \ref{anti1}, we obtain
\[ g*h*\Delta(u_{i})+\Delta(g*h)*u_{i}=h*\Delta(u_{i}*g)+\Delta(h)*u_{i}*g.\]
By this identity and using similar arguments as above it
follows that
\begin{equation}\label{anti3}
\Delta(f*g)=g*\Delta(f)+\Delta(g)*f-f*g*\xi
\end{equation} 
for all $f,g,h \in \g$. Hence from \ref{anti2} and \ref{anti3}, for each $f,g,h \in \g$, we find that $\mu*f*g=f*g*\xi$. So by Cohen's factorization theorem and Lemma \ref{zz}-$(ii)$, it follows that $\xi \in \mathcal{Z}(\m)$. Define $D:\g\rightarrow \m$ by $D(f)=\Delta(f)-\xi*f$. By \ref{anti2} and the fact that $\xi \in \mathcal{Z}(\m)$, it follows that $D$ is an Jordan derivation. From Cohen's factorization theorem and \ref{anti2}, we obtain $\Delta(\g)\subseteq \g$ and hence $D(\g)\subseteq \g$. Since $\g$ is semisimple, it follows that $D$ is a derivation \cite{sin}. So by derivation problem \cite{lo}, there is a $\mu \in \m$, such that $D(f)=f*\mu-\mu*f$ for all $f\in \g$. Letting $\nu=\mu-\xi$. So $\xi=\mu-\nu \in \mathcal{Z}(\m)$ and $\Delta(f)=f*\mu-\nu*f$ for all $f\in \g$.
\par 
Now by \ref{anti3} and the fact that $D$ is a derivation we see that 
\begin{equation*}
\begin{split}
\Delta(f*g)+f*g*\xi&=g*\Delta(f)+\Delta(g)*f \\ &=
g*D(f)+g*\xi*f+D(g)*f+\xi*g*f\\
&= D(g*f)+2g*f*\xi\\ &=
\Delta(g*f)+g*f*\xi,
\end{split}
\end{equation*} 
for all $f,g\in \g$. So
\[ f*g*\mu-\nu*f*g+f*g*\xi=g*f*\mu-\nu*g*f+g*f*\xi,\]
and hence
\[ f*g*\mu- \mu *f*g+2f*g*\xi=g*f*\mu-\mu*g*f+2g*f*\xi,\] 
for all $f,g\in \g$. Therefore
\[ [[f,g],\mu]+2 [f,g]*(\mu-\nu)=0, \]
for all $f,g\in \g$.
\par 
$(ii)$ In order to prove this we consider the contiuous bilinear map $\phi:\g \times \g \rightarrow \m$ by $\phi(f,g)=\Delta(g^{\star})^{\star}*f+g*\Delta(f)$. If $f,g\in \g$ are such that $f*g=0$, then $\phi(f,g)=0$. So by Lemma \ref{bi}-$(i)$, we get $ \phi (f*g,h)=\phi (f,g*h)$ for all $f,g,h \in \g$. Therefore
\begin{equation}\label{anti4}
\Delta(h^{\star})^{\star}*f*g+h*\Delta(f*g)=\Delta(h^{\star}*g^{\star})^{\star}*f+g*h*\Delta(f),
 \end{equation}
 for all $f,g,h \in \g$. Setting $f=u_{i}$ in \ref{anti4} and by using similar methods as part (i), we get
 \begin{equation}\label{anti5}
\Delta(h^{\star})^{\star}*g+h*\Delta(g)=\Delta(h^{\star}*g^{\star})^{\star}+g*h*\xi,
 \end{equation}
 for all $g,h \in \g$, where $\xi\in \m$ and $\Delta(u_{i})\xlongrightarrow{w^{*}} \xi$. By \ref{anti5} we have
 \begin{equation}\label{anti55} g^{\star}*\Delta(h^{\star})+\Delta(g)^{\star}*h^{\star}=\Delta(h^{\star}*g^{\star})+\xi^{\star}*h^{\star}*g^{\star},\end{equation}
 for all $g,h \in \g$. Letting $h^{\star}=u_{i}$, we arrive at
 \[ g^{\star}*\xi+\Delta(g)^{\star}=\Delta(g^{\star})+\xi^{\star}*g^{\star},\]
 for all $g \in \g$. Hence
 \begin{equation}\label{anti6}
\Delta(g^{\star})-g^{\star}*\xi=(\Delta(g)-g*\xi)^{\star},
 \end{equation} 
 for all $g \in \g$. From \ref{anti55} we have
 \begin{equation}\label{anti555} 
 \Delta(f*g)=g*\Delta(f)+\Delta(g^{\star})^{\star}*f-\xi^{\star}*f*g,
 \end{equation}
 for all $f,g \in \g$.
  Define $D:\g\rightarrow \m$ by $D(f)=\Delta(f)-f*\xi$. By \ref{anti6} and \ref{anti555}, it follows that $D$ is a $\star$-Jordan derivation and $D(\g)\subseteq \g$. Hence it is a $\star$-derivation and so there is a $\mu\in \m$ with $Re\mu\in \mathcal{Z}(\m)$, such that $D(f)=f*\mu-\mu*f$ for all $f\in \g$. 
  \par 
  Now by \ref{anti555} and the fact that $D$ is a $\star$-derivation, we have 
  \[ \Delta(f*g)+\xi^{\star}*f*g=\Delta(g*f)+\xi^{\star}*g*f,\]
for all $f,g \in \g$. So 
\[ f*g*\mu-\mu*f*g+\xi^{\star}*f*g=g*f*\mu-\mu*g*f+\xi^{\star}*g*f,\]
and hence
\[ [[f,g],\mu]+\xi^{\star}*[f,g]+[f,g]\xi=0,\]
for all $f,g \in \g$. By setting $\nu=\mu+\xi$, we have $\Delta(f)=f*\nu-\mu*f$ and 
\[ [[f,g],\mu]+(\nu-\mu)^{\star}*[f,g]+[f,g]*(\nu-\mu)=0,\]
for all $f,g \in \g$, where $Re\mu\in \mathcal{Z}(\m)$.
\par 
$(iii)$ Consider the map $D:\g\rightarrow \m$ defined by $D(f)=\Delta(f^{\star})^{\star}$. It is easily seen that the map $D$ satisfies the conditions in part $(ii)$. So, there exists $\mu_{1}, \nu_{1}\in \m$ such that $D(f)=f*\nu_{1}-\mu_{1}*f$ for all $f \in \g$ with $Re\mu_{1}\in \mathcal{Z}(\m)$ and 
\[ [[f,g],\mu_{1}]+(\nu_{1}-\mu_{1})^{\star}*[f,g]+[f,g]*(\nu_{1}-\mu_{1})=0,\]
for all $f,g \in \g$. Then $\Delta(f)=f*\mu-\nu*f$ for all $f\in \g$, where $\nu=-\nu_{1}^{\star}$, $\mu=-\mu_{1}^{\star}$ with $Re\mu\in \mathcal{Z}(\m)$ and
\[ [[f,g],\mu]+[f,g]*(\mu-\nu)^{\star}+(\mu-\nu)[f,g]=0,\]
for all $f,g \in \g$ .
\end{proof}
Note that the converse of Theorems \ref{o1}-$(i)-(iii)$ and \ref{anti}-$(i)-(iii)$ hold. It is checked easily.
\section{Derivations through two-sided orthogonality conditions}
In this section we will consider a linear map $\Delta:\g \rightarrow \m$ behaving like derivation at two-sided orthogonality conditions, where $G\in [SIN]$.
\begin{thm}\label{two}
Let $G$ be a locally compact group with $G\in [SIN]$, and let $\Delta:\g \rightarrow \m $ be a continuous linear map.
\begin{enumerate} 
 \item[(i)] Assume that 
\[ f*g=g*f=0 \Longrightarrow f*\Delta(g)+\Delta(f)*g=g*\Delta(f)+\Delta(g)*f=0\quad (f,g \in \g).\]
Then there are measures $\mu,\nu \in \m$ such that $\Delta(f)=f*\mu-\nu*f$ for all $f\in \g$, where $\mu-\nu\in \mathcal{Z}(\m)$.
\item[(ii)] Assume that 
\begin{equation*}\label{1o2}
 f*g^{\star}=g^{\star}*f=0 \Longrightarrow f*\Delta(g)^{\star}+\Delta(f)*g^{\star}=\Delta(g)^{\star}*f+g^{\star}*\Delta(f)=0\quad (f,g \in \g).
 \end{equation*}
Then there are $\mu, \nu\in \m$ such that $\Delta(f)=f*\mu-\nu*f$ for all $f \in \g$, where $Re\mu\in \mathcal{Z}(\m)$ and $Re(\mu-\nu)\in \mathcal{Z}(\m)$.
\end{enumerate}
\end{thm}
\begin{proof}
Since $G\in [SIN]$, we can suppose that $(u_i)_{i\in I}$ is a central bounded approximate identity of $\g$ such that $u_{i}\xlongrightarrow{w^{*}} \delta_{e}$, where $\delta_{e}$ is the identity of $\m$. Also  we can assume that $\Delta(u_{i})\xlongrightarrow{w^{*}} \xi$ with $\xi \in \m$, since the net $(\Delta(u_{i}))_{i\in I}$ is bounded.
\par 
$(i)$
Define a continuous bilinear map $\phi:\g \times \g \rightarrow \m$ by $\phi(f,g)=f*\Delta(g)+\Delta(f)*g$. So $\phi(f,g)=0$, whenever $f*g=g*f=0$. Hence by Lemma \ref{bi}-$(ii)$, we get $ \phi (f*g,h)=\phi (f,g*h)$ for all $f,h \in \mathcal{Z}(\g)$ and $g\in \g$. Therefore
\begin{equation}\label{tw1}
f*g*\Delta(h)+\Delta(f*g)*h=f*\Delta(g*h)+\Delta(f)*g*h,
\end{equation}
for all $f,h \in \mathcal{Z}(\g)$ and $g\in \g$. Taking $f=u_{i}$ in \eqref{tw1}, we get
 \[u_{i}*g*\Delta(h)+\Delta(u_{i}*g)*h=u_{i}*\Delta(g*h)+\Delta(u_{i})*g*h,\]
 for all $h \in \mathcal{Z}(\g)$ and $g\in \g$. From continuity of $\Delta$, we get $u_{i}*\Delta(g*h)+\Delta(u_{i})*g*h$ converges to $g*\Delta(h)+\Delta(g)*h$ with respect to the norm topology. On the other hand, from separately $w^{*}$-continuity of convolution product in $\m$ and $\Delta(u_{i})\xlongrightarrow{w^{*}} \xi$, it follows that $u_{i}*\Delta(g*h)+\Delta(u_{i})*g*h \xlongrightarrow{w^{*}} \Delta(g*h)+\xi*g*h$. Hence 
\begin{equation}\label{tw2}
\Delta(g*h)+\xi*g*h=g*\Delta(h)+\Delta(g)*h,
 \end{equation}
 for all $h \in \mathcal{Z}(\g)$ and $g\in \g$. Now letting $h=u_{i}$ in \ref{tw2} and similar as above we get $\xi*g=g*\xi$ for all $g\in \g$. From Lemma \ref{zz}-$(ii)$, it follows that $g\in \mathcal{Z}(\m)$. Define $D:\g\rightarrow \m$ by $D(f)=\Delta(f)-\xi*f$. Then $D$ is a continuous linear map which
satisfies 
\begin{equation}\label{tw3}
 f*g=g*f=0 \Longrightarrow f*D(g)+D(f)*g=g*D(f)+D(g)*f=0\quad (f,g \in \g),
 \end{equation}
 and $D(u_{i})\xlongrightarrow{w^{*}} 0$. We will show that $D$ is a  derivation. In order to prove this we consider the continuous bilinear map $\phi:\g \times \g \rightarrow \m$ by $\phi(f,g)=f*D(g)+D(f)*g$. If $f,g\in \g$ are such that $f*g=g*f=0$, then \eqref{tw3} gives $\phi(f,g)=0$. So by Lemma \ref{bi}-$(ii)$, we get 
 \[\phi (f*g,h*k)-\phi(f,g*h*k)+\phi (k*f,g*h)-\phi(k*f*g,h),\]
 for all $f,g,h \in \g$. Therefore
 \begin{equation}\label{tw4}
 \begin{split}
&f*g*D(h*k)+D(f*g)*h*k-f*D(g*h*k)-D(f)*g*h*k+\\ &
k*f*D(g*h)+D(k*f)*g*h-k*f*g*D(h)-D(k*f*g)*h=0,
\end{split}
 \end{equation}
 for all $f,g,h,k \in \g$. Taking $f=u_{i}$ in \eqref{tw4}, since $D(u_{i})\xlongrightarrow{w^{*}} 0$, it follows that
 \begin{equation}\label{tw5}
 g*D(h*k)+D(g)*h*k-D(g*h*k)+k*D(g*h)+D(k)*g*h-k*g*D(h)-D(k*g)*h=0,
 \end{equation}
 for all $g,h,k \in \g$. Now letting $h=u_{i}$ in \eqref{tw5}, we get 
  \begin{equation}\label{tw6}
 g*D(k)+D(g)*k-D(g*k)+k*D(g)+D(k)*g-D(k*g)=0,
 \end{equation}
 for all $g,k \in \g$. So $D$ is a Jordan derivation. An analogue of the Cohen factorization theorem on Banach algebras holds for Jordan-Banach algebras (see \cite{ak1, ak2}). This result implies that for any $h\in \g$ there exist $f,g \in \g$ such that $h=f*g*f$. Since $D$ is a Jordan derivation, it follows that 
 \[ D(h)=D(f*g*f)=D(f)*g*f+f*D(g)*f+f*g*D(f),\]
 for all $h\in \g$, where $h=f*g*f$. Thus $D(\g)\subseteq \g$ and hence $D$ is a derivation. So by derivation problem there is a $\mu\in \m$ such that $D(f)=f*\mu-\mu*f$ for all $f\in \g$. Setting $\nu=\mu-\xi$. So $\Delta(f)=f*\mu-\nu*f$ for all $f \in \g$ and $\mu-\nu\in \mathcal{Z}(\m)$.
 \par 
$(ii)$ Define continuous bilinear maps $\phi, \psi:\g \times \g \rightarrow \m$ by \[\phi(f,g)=f*\Delta(g^{\star})^{\star}+\Delta(f)*g \,\, \text{and} \,\, \psi(f,g)=\Delta(g^{\star})^{\star}*f+g*\Delta(f),\]
for all $f,g \in \g$. It is straightforward to check that $\phi(f,g)=0$ and $\psi(f,g)=0$, whenever $f*g=g*f=0$. Thus by Lemma \ref{bi}-$(ii)$, we get 
\begin{equation}\label{tw7}
f*g*\Delta(h^{\star})^{\star}+\Delta(f*g)*h=f*\Delta(h^{\star}*g^{\star})^{\star}+\Delta(f)*g*h,
\end{equation}
and
\begin{equation}\label{tw8}
\Delta(h^{\star})^{\star}*f*g+h*\Delta(f*g)=\Delta(h^{\star}*g^{\star})^{\star}*f+g*h*\Delta(f),
\end{equation}
for all $f,h \in \mathcal{Z}(\g)$ and $g\in \g$. Letting $f=u_{i}$ in \eqref{tw7} and \eqref{tw8}, and applying $u_{i}\xlongrightarrow{w^{*}} \delta_{e}$, $\Delta(u_{i})\xlongrightarrow{w^{*}} \xi$ we obtain
\begin{equation*}
g*\Delta(h^{\star})^{\star}+\Delta(g)*h=\Delta(h^{\star}*g^{\star})^{\star}+\xi*g*h,
\end{equation*}
and
\begin{equation*}
\Delta(h^{\star})^{\star}*g+h*\Delta(g)=\Delta(h^{\star}*g^{\star})^{\star}+g*h*\xi,
\end{equation*}
for all $h \in \mathcal{Z}(\g)$ and $g\in \g$. We now apply $\star$, set $h^{\star}=u_{i}$, to see that
\begin{equation}\label{tw9}
\xi*g^{\star}+\Delta(g)^{\star}=\Delta(g^{\star})+g^{\star}*\xi^{\star},
\end{equation}
and
\begin{equation}\label{tw10}
g^{\star}*\xi+\Delta(g)^{\star}=\Delta(g^{\star})+\xi^{\star}*g^{\star},
\end{equation}
for all $g\in \g$. From equations \eqref{tw9} and \eqref{tw10}, we get $(\xi+\xi^{\star})*g^{\star}=g^{\star}*(\xi^{\star}+\xi)$ for all $g\in \g$. Hence $Re\xi \in \mathcal{Z}(\m)$.
\par 
Define $D:\g\rightarrow \m$ by $D(f)=\Delta(f)-\xi*f$. Then $D$ is a continuous linear map and by \eqref{tw9}, we  have $D(f^{\star})=D(f)^{\star}$ for all $f\in\g$. Hence $D$ is a $\star$-map. If $f*g=g*f=0$, then by hypothesis, definition of $D$ and the fact that $D$ is a $\star$-map and $Re\xi \in \mathcal{Z}(\m)$, we have
\begin{equation*}
\begin{split}
f*D(g)+D(f)*g&=f*D(g^{\star})^{\star}+D(f)*g\\ &=
f*(\Delta(g^{\star})-\xi*g^{\star})^{\star}+(\Delta(f)-\xi*f)*g=0,
\end{split}
\end{equation*}
and 
\begin{equation*}
\begin{split}
g*D(f)+D(g)*f&=g*D(f)+D(g^{\star})^{\star}*f\\ &=
g*(\Delta(f)-\xi*f)+(\Delta(g^{\star})-\xi*g^{\star})^{\star}*f\\&=
-g*\xi*f-g*\xi^{\star}*f=-g*f*(\xi+\xi^{\star})=0.
\end{split}
\end{equation*}
So $D$ satisfies in $(i)$ and hence there are $\mu,\nu \in \m$ such that $D(f)=f*\mu-\nu_{1}*f$ for all $f\in \g$. Since $D(u_{i})\xlongrightarrow{w^{*}} 0$ , it follows that $\mu=\nu_{1}$. On the other hand $D$ is a $\star$-map, so by Remark \ref{der}, $D(f)=f*\mu-\mu*f$ for all $f\in \g$, where $Re\mu\in \mathcal{Z}(\m)$. Therefore by letting $\nu=\mu-\xi$ we have $\Delta(f)=f*\mu-\nu*f$ for all $f\in \g$, where $Re\mu\in \mathcal{Z}(\m)$ and $Re(\mu-\nu)\in \mathcal{Z}(\m)$.
\end{proof}
Note that the converse of Theorems \ref{two}-$(i)-(iii)$ holds. It is checked easily.


\bibliographystyle{amsplain}
\bibliography{xbib}

\end{document}